\documentclass[a4paper,11pt,reqno,noindent]{amsart}
\usepackage[centertags]{amsmath}
\usepackage{amsfonts,amssymb,amsthm,amscd,dsfont,cases,esint,enumerate,xcolor}
\usepackage[T1]{fontenc}
\usepackage[english]{babel}
\usepackage[applemac]{inputenc}
\usepackage[body={15cm,21.5cm},centering]{geometry} 
\usepackage{fancyhdr}
\numberwithin{equation}{section}

\theoremstyle{plain}
\newtheorem{theor10}{Theorem}
\newenvironment{theor1}
  {\pushQED{\qed}\begin{theor10}}
  {\popQED\end{theor10}}
\newtheorem{prop10}[theor10]{Proposition}
\newenvironment{prop1}
  {\pushQED{\qed}\begin{prop10}}
  {\popQED\end{prop10}}
\newtheorem{cor10}[theor10]{Corollary}

\newtheorem{lem10}[theor10]{Lemma}

\newtheorem{theor0}{Theorem}[section]

\newtheorem{lem0}[theor0]{Lemma}
\newenvironment{lem}
  {\pushQED{\qed}\begin{lem0}}
  {\popQED\end{lem0}}
\newtheorem{prop0}[theor0]{Proposition}

\newtheorem{cor0}[theor0]{Corollary}

\theoremstyle{definition}
\newtheorem{rems0}[theor0]{Remarks}

\newtheorem{rem0}[theor0]{Remark}

\theoremstyle{plain}
\newtheorem{as0}[theor0]{Assumption}

\newtheorem*{asn0*}{\assumptionnumber}
  \providecommand{\assumptionnumber}{}
\makeatletter
\newenvironment{asn0}[2]
   {\renewcommand{\assumptionnumber}{Assumption \!#1 {\normalfont--- #2}}
    \begin{asn0*}
    \protected@edef\@currentlabel{{\normalfont#1}}}
   {\end{asn0*}}
\makeatother

\makeatletter
\newenvironment{asn01}[1]
   {\renewcommand{\assumptionnumber}{Assumption \!#1}
    \begin{asn0*}
    \protected@edef\@currentlabel{{\normalfont#1}}}
   {\end{asn0*}}
\makeatother

\newcommand{\N}{\mathbb N}

\newcommand{\Pc}{\mathcal{P}}

\newcommand{\R}{\mathbb R}
\newcommand{\Z}{\mathbb Z}
\newcommand{\Aa}{A}

\newcommand{\E}{\mathbb{E}}

\newcommand{\step}[1]{\noindent \textit{Step} #1.}

\newcommand{\pr}[1]{\mathbb{P}\left[ #1 \right]}

\newcommand{\expec}[1]{\mathbb{E}\left[ #1 \right]}
\newcommand{\expech}[1]{\mathbb{E}_h\left[ #1 \right]}
\newcommand{\expecm}[1]{\mathbb{E}\big[ #1 \big]}

\usepackage[colorlinks,citecolor=black,urlcolor=black]{hyperref}

\title[Gevrey regularity of the diffusion coefficient]{A short proof of Gevrey regularity for homogenized coefficients of the Poisson point process}

\author[M. Duerinckx]{Mitia Duerinckx}
\address[Mitia Duerinckx]{Universit\'e Paris-Saclay, CNRS, Laboratoire de Math\'ematiques d'Orsay, 91405~Orsay, France \& Universit\'e Libre de Bruxelles, D\'epartement de Math\'ematique, 1050~Brussels, Belgium}
\email{mduerinc@ulb.ac.be}
\author[A. Gloria]{Antoine Gloria}
\address[Antoine Gloria]{Sorbonne Universit\'e, CNRS, Universit\'e de Paris, Laboratoire Jacques-Louis Lions, 75005~Paris, France \& Institut Universitaire de France \& Universit\'e Libre de Bruxelles, D\'epartement de Math\'ematique, 1050~Brussels, Belgium}
\email{gloria@ljll.math.upmc.fr}

\begin{document}
\selectlanguage{english}

\maketitle

\begin{abstract}
In this short note we capitalize on and complete our previous results on the regularity of the homogenized coefficients for Bernoulli perturbations by addressing the case of the Poisson point process, for which the crucial uniform local finiteness assumption fails. 
In particular, we strengthen the qualitative regularity result first obtained in this setting by the first author to Gevrey regularity of order~2.
The new ingredient is a fine application of properties of Poisson point processes, in a form recently used by Giunti, Gu, Mourrat, and Nitzschner.

\bigskip\noindent
{\sc MSC-class:} 35R60, 60G55.

\noindent{\sc keywords:} homogenized coefficients, Poisson point process.
\end{abstract}

\setcounter{tocdepth}{2}

\section{Introduction and main result}

\subsection{Context}

This short note is concerned with the expansion of the homogenized coefficients under Bernoulli perturbations of Poisson point processes, and can be considered as an appendix to \cite{DG-16a}. Consider a locally finite stationary ergodic random point set $\Pc=\{x_n\}_n$ in $\R^d$ ($d\ge 1$), to which we associate the random (diffusion) coefficient field $A(\Pc)$ on $\R^d$
\begin{equation}\label{e.coeff}
A(\Pc)(x):=A_1(x) +(A_2(x)-A_1(x)) \mathds{1}_{\cup_n B(x_n)}(x),
\end{equation}
where $B(x_n)$ denotes the unit ball centered at point $x_n$, and $A_1$ and $A_2$ are ergodic stationary random uniformly elliptic symmetric coefficient fields (that is, the standard assumptions of stochastic homogenization). Symmetry is not essential in what follows, see e.g.~the discussion at the end of \cite[Section~1]{DG-16a}. Since $A_1$, $A_2$, and $\Pc$ are stationary and ergodic, the random coefficient field $A(\Pc)$ is also stationary and ergodic itself, and we can define the associated homogenized coefficient $\bar A(\Pc)$, a deterministic matrix given in direction $e \in \R^d$ by
\begin{equation}\label{e.coeff-hom}
\bar A (\Pc)e = \expec{A(\nabla \varphi+e)},
\end{equation}
where $\varphi$ is the so-called corrector, see~\eqref{e.corr} below for details, and where $\expec{\cdot}$ denotes the expectation in the underlying probability space.

For all $0\le p \le 1$, denote by $\Pc^{(p)}$ the random Bernoulli deletion of $\Pc$,
that is, $\Pc^{(p)}=\{x_n: b_n^{(p)}=1\}$ with $\{b_n^{(p)}\}_n$ a sequence of independent Bernoulli variables of law $(1-p) \delta_0+ p\delta_1$. This means that $\Pc^{(p)}$ is a decimated point process (with $\Pc^{(0)}=\varnothing$ and $\Pc^{(1)}=\Pc$).
With~$\Pc^{(p)}$, we associate $A^{(p)}:=A(\Pc^{(p)})$ and $\bar A^{(p)}:=\bar A(\Pc^{(p)})$ as in \eqref{e.coeff} and \eqref{e.coeff-hom}.
In these terms, we are interested in the regularity of the map $p \mapsto \bar A^{(p)}$.
Inspired by \cite{Anantharaman-LeBris-11,Anantharaman-LeBris-12},  we established in~\cite{DG-16a} its analyticity under the crucial assumption 
that $\Pc$ be uniformly locally finite, that is, if $\sup_{x\in \R^d} \sharp\{x_n \in B(x)\}<\infty$. 
This result, which does not rely on any mixing assumption of $\Pc$ itself (besides qualitative ergodicity), does not apply to the Poisson point process since the latter is not uniformly locally finite. 

The present note is concerned with the case of the Poisson point process. Denote by $\Pc_{\lambda}$  a Poisson point process with intensity $\lambda\ge 0$ (that is, $\expecm{ \sharp\{\Pc_\lambda \cap [0,1)^d\}}=\lambda$).
Then, the decimated process $\Pc_\lambda^{(p)}$ has the same law as $\Pc_{p\lambda}$, so that the regularity of $\lambda\mapsto\bar A_\lambda$ is equivalent to the regularity of $p\mapsto \bar A^{(p)}_\lambda$ for fixed $\lambda$.
As announced in \cite{DG-16a}, exploiting that $\Pc_\lambda$ has finite range of dependence, and assuming that $A_1$ and $A_2$ are constant, the first author proved the smoothness of $\lambda \mapsto \bar A_{\lambda}$ in his PhD thesis~\cite[Theorem 5.A.1]{D-thesis},
based on the quantitative homogenization estimates of~\cite{Gloria-Otto-10b} (in the spirit of~\cite{Mourrat-13} for the first-order expansion in the discrete setting).
The question of quantitative smoothness (such as Gevrey regularity or analyticity) of $\lambda \mapsto \bar A_\lambda$ was left open.

Motivated by applications to homogenization of particle systems~\cite{giunti2021quantitative},
Giunti, Gu, Mourrat, and Nitzschner recently addressed
a related  problem in a different setting, and proved the Gevrey regularity of $\lambda \mapsto \bar a_\lambda$ in  \cite{giunti2021smoothness} (a variant of $\lambda \mapsto \bar A_\lambda$). Their approach is based on Poisson calculus (cf.~\cite{LP-18}), which they use both to derive formulas and to prove estimates. In the introduction of \cite{giunti2021smoothness}, the authors point out that the strategy they use could be applied to prove the regularity of $\lambda \mapsto \bar A_\lambda$ in our setting.

\subsection{Main results}
The aim of this note is twofold. First, we show that \cite{giunti2021smoothness}, besides having the same layout as \cite{DG-16a}, is a direct implementation (in the setting of homogenization of particle systems and Poisson calculus) of the strategy based on the triad ``\emph{local approximation / cluster expansions / improved~$\ell^1-\ell^2$ estimates}'' that we introduced for general point processes in~\cite{DG-16a}.
Second, we single out the new ingredient of \cite{giunti2021smoothness} for Poisson processes wrt to \cite{DG-16a}  (see Lemma~\ref{lem:JC} below), and combine it with our general formulation of \cite{DG-16a} in order to prove the Gevrey regularity of the map $\lambda \mapsto \bar A_\lambda$
 with little effort.

\medskip

We start with the comparison of \cite{giunti2021smoothness} with our previous work~\cite{DG-16a}. Although the precise functional setting is different, the identities, the estimates, and the arguments leading to them have the same form.
The arguments in~\cite{giunti2021smoothness} are as follows:
\begin{enumerate}[$\bullet$]
\item The authors first introduce in~\cite[Section~3]{giunti2021smoothness} a sequence of local approximations of $\bar a_\lambda$ 
 on bounded domains. This is in line with the massive approximations used in~\cite{DG-16a} (making cluster expansions finite).
\item They view correctors as functions of sets of indices and introduce a difference calculus (see \cite[(2.9)--(2.11) and Proposition~5.1]{giunti2021smoothness}) that provides a natural way to write cluster expansions. This coincides with the point of view and the definitions of~\cite[Section~2.2]{DG-16a}.
\item They dedicate~\cite[Section~4]{giunti2021smoothness} to the proof of $C^{1,1}$-regularity to illustrate the general strategy, as we did in \cite[Section~3]{DG-16a} for the map $p \mapsto \bar A^{(p)}$.
\item They turn in~\cite[Section~5]{giunti2021smoothness} to the proof of their main result~\cite[Theorem~2.3]{giunti2021smoothness}, which they split into several parts:
\begin{enumerate}[---]
\item They first derive explicit formulas \cite[(5.9)--(5.13)]{giunti2021smoothness} for the terms of the cluster expansion and for the remainder. These are to be compared to \cite[Lemma~5.1]{DG-16a} (once reformulated using Poisson calculus).
\item Then they introduce and prove ``key estimates''   in~\cite[Proposition~5.4]{giunti2021smoothness}. Both the statement and the proof are to be compared to what is called ``improved $\ell^1-\ell^2$ estimates'' in the original~\cite[Proposition~4.6]{DG-16a}, the very core of \cite{DG-16a} (reformulated using Poisson calculus again).  
The only new ingredient with respect to the strategy of \cite{DG-16a} turns out to be an interesting property of the Poisson point process, which we single out here in Lemma~\ref{lem:JC} below.
\item Finally, they combine the explicit formulas for the cluster expansion and remainder together with the $\ell^1-\ell^2$ estimates in order to pass to the limit in the approximation parameter, cf.~\cite[Section~5.5]{giunti2021smoothness}. This string of arguments is similar to~\cite[Section~5]{DG-16a}.
They remark that a careful tracking of the constants in their proofs (which they omit) would reveal that $\lambda \mapsto \bar a_\lambda$ has Gevrey regularity of order 2.
\end{enumerate}
\end{enumerate}
Next, relying on our original results in~\cite{DG-16a} (without using Poisson calculus),
together with a few adaptations,
we shall establish the following version of \cite[Theorem~2.1]{DG-16a} for the Poisson point process.
\begin{theor1}\label{th:main}
The map $\lambda \mapsto \bar A_\lambda$ is Gevrey regular of order $2$ on $[0,\infty)$, and derivatives
are given by cluster formulas as in~\cite{DG-16a}.
\end{theor1}
Compared to our previous result \cite[Theorem~2.1]{DG-16a}, Theorem~\ref{th:main} treats the Poisson point process (relaxing the assumption that the point process be uniformly locally finite). This comes at a price:   whereas real-analyticity was established in \cite[Theorem~2.1]{DG-16a}, we only obtain Gevrey-regularity in Theorem~\ref{th:main}. Also, it is not clear to us to what extent this result is expected to hold for the thinning by Bernoulli deletion of other non-uniformly locally finite point processes than Poisson.

\section{Proof of the Gevrey regularity}

\subsection{Strategy of the proof}\label{sec:strategy}

Recall that $\Pc_{\lambda}^{(p)}$ and $\Pc_{p\lambda}$ have the same law for all $p\in[0, 1]$, hence ${\bar A}_\lambda^{(p)}={\bar A}_{p\lambda}$,
which entails that regularity of $\lambda\mapsto\bar A_\lambda$ on $[0,\infty)$ is equivalent to regularity of $p \mapsto \bar A_\lambda^{(p)}$ for any~$\lambda>0$.
In addition, replacing the underlying random field $A_1$ by the law of $A_\lambda^{(p_0)}$ turns $A_\lambda^{(p)}$ into the law of $A_\lambda^{(p+p_0)}$, hence we may restrict to proving the regularity of $p \mapsto \bar A_\lambda^{(p)}$ at $p=0$ for any~$\lambda>0$.
In what follows, 
we let $\lambda>0$ be arbitrary, yet fixed, and we skip the subscript $\lambda$ for simplicity.
We start with two approximations.
First, as in \cite{DG-16a}, we replace the corrector gradient $\nabla \varphi$,
that is the centered stationary gradient solution of the 
whole-space PDE
\begin{equation}\label{e.corr}
-\nabla \cdot A(\nabla \varphi+e)=0,
\end{equation}
by the gradient $\nabla \varphi_T$ of its massive approximation, that is
the corresponding solution of the 
whole-space PDE
\begin{equation}\label{e.corrT}
\frac1T \varphi_T-\nabla \cdot A(\nabla \varphi_T+e)=0.
\end{equation}
As opposed to \eqref{e.corr}, the latter equation~\eqref{e.corrT} is well-posed on a deterministic level (that is, well-posed for any uniformly elliptic coefficient field $A$), and the dependence of $\nabla \varphi_T(x)$ upon the values of $A$ restricted on $Q(y)=[y,y+1)^d$ is uniformly exponentially small in
$\frac{|x-y|}{\sqrt T}$.
Next, we replace the Poisson point process $\Pc$ by a sequence of uniformly locally finite point processes $\{\Pc_h\}_h$ defined as follows.
For $h>0$, we decompose $\R^d$ into the union of cubes $Q_h(z)=z+[0,h)^d$ with $z\in (h\Z)^d$.
On each cube $Q_h(z)$ we pick randomly a point $x_z$ (independently of the others), we attach an independent Bernoulli variable $b_z$ of parameter $\lambda h^d$, and finally set
$$
\Pc_h:=\{x_z \,:\, z \in (h\Z)^d,b_z=1\}.
$$
So defined, $\Pc_h$ is indeed uniformly locally finite and it has $h$-discrete stationarity and finite range of dependence. In addition, $\Pc_h$ converges in law to $\Pc$ as $h\downarrow 0$.
Using these two approximations, we introduce the following  proxy for the homogenized coefficients,
$$
\bar A_{T,h}^{(p)}e := \expech{  A(\Pc_h^{(p)})(\nabla \varphi_{T,h}^{(p)}+e)},
$$
with the short-hand notation $\expech{\cdot}:=\E[\fint_{Q_h(0)} \cdot]$ and $\varphi_{T,h}^{(p)}=\varphi_{T}(\Pc_h^{(p)})$.
By qualitative stochastic homogenization arguments (see e.g.~\cite[Theorem~1]{Gloria-12} for the convergence in $T$ and \cite[Step~1 in Section~5.2]{DG-16a} for the convergence in $h$), we have for all $p\in[0,1]$,
\begin{equation}\label{e.qual-conv}
\lim_{T\uparrow \infty,h \downarrow 0} \bar A_{T,h}^{(p)} = \bar A^{(p)}.
\end{equation}
By~\cite[Theorem~2.1]{DG-16a}, $p \mapsto \bar A_{T,h}^{(p)}$ is real-analytic close to zero (and actually on the whole interval~$[0,1]$), and there exists a sequence $\{\bar A_{T,h}^{j}\}_j$, given by explicit cluster formulas, cf.~Lemma~\ref{lem:cluster-form} below,
such that for all $p$ small enough we have
\begin{equation}\label{eq:expansion-ATh}
\bar A_{T,h}^{(p)}\,=\, \sum_{j=0}^\infty \frac{p^j}{j!} \bar A_{T,h}^{j}.
\end{equation}
As we shall see, Theorem~\ref{th:main} follows in the limit $T\uparrow \infty,h \downarrow 0$ provided we prove that there exists $C<\infty$  such that this sequence further satisfies for all $j$,
\begin{equation}\label{e.unif-bd}
\sup_{T\ge 1, h \le 1} |\bar A_{T,h}^{j}| \,\le\,j!^2C^j.
\end{equation}
The main ingredient to  \eqref{e.unif-bd} is Proposition~\ref{prop:main} below.
Before we state this result, let us recall some notation and results borrowed from~\cite{DG-16a}.

\subsection{Difference operators and inclusion-exclusion formula}
We start by considering correctors as functions of indices, and then recall the associated difference calculus and the inclusion-exclusion formula.
In what follows, we write $\Pc=\{x_n\}_n$ and set $J_n:=B(x_n)$. Note that inclusions $\{J_n\}_n$ could have different shapes and even be random as well provided they are uniformly bounded.

\medskip
\noindent
{\bf Correctors as functions of indices.}

For all (possibly infinite) subsets $E\subset\N$, we define $A^E:=A_1+C^E$, where $C^E:=(A_2-A_1)\mathds1_{J^E}$ and $J^E:=\bigcup_{n\in E}J_n$,
and we introduce the following variant of \eqref{e.corrT}:
\begin{equation}\label{eq:modif-corr}
\frac1T\varphi_{T}^E-\nabla\cdot A^E(\nabla\varphi_{T}^E+e)=0.
\end{equation}
Setting $E^{(p)}:=\{n\in\N\,:\,b_n^{(p)}=1\}$, we use the short-hand notation $C^{(p)}:=C^{E^{(p)}}$, $A^{(p)}:=A^{E^{(p)}}$, and $\varphi_T^{E^{(p)}}=\varphi_T^{(p)}$.

\medskip
\noindent
{\bf Difference operators.}

We introduce for all $n\in \N$ a difference operator $\delta^{\{n\}}$ acting generically on measurable functions of the point process, and in particular on approximate correctors as follows: for all $H\subset \N$,
\[\delta^{\{n\}}\varphi_T^H:=\varphi_T^{H\cup\{n\}}-\varphi_T^H.\]
This operator yields a natural measure of the sensitivity of the corrector $\varphi_T^H$ with respect to the perturbation of the medium at inclusion $J_n$. 
For all finite $F\subset\N$, we further introduce the higher-order difference operator $\delta^F=\prod_{n\in F}\delta^{\{n\}}$. More explicitly, this difference operator $\delta^F$ acts as follows on approximate correctors $\varphi_T^H$: for all $H\subset \N$,
\begin{equation}\label{eq:def-diff}
\delta^{F}\varphi_T^H:=\sum_{l=0}^{|F|}(-1)^{|F|-l}\sum_{G\subset F\atop |G|=l}\varphi_T^{G\cup H}=\sum_{G\subset F}(-1)^{|F\setminus G|}\varphi_T^{G\cup H},
\end{equation}
with the convention $\delta^\varnothing\varphi_T^H=(\varphi_T^H)^\varnothing:=\varphi_T^H$.
As in the physics literature, see~\cite{Torquato-02}, such operators are used
to formulate {\it cluster expansions}, which are viewed as formal proxies for Taylor expansions with respect to the Bernoulli perturbation: up to order $k$ in the parameter $p$, the cluster expansion for the perturbed corrector reads, for small $p\ge0$,
$$
\varphi_T^{(p)}\leadsto  \varphi_T+\sum_{n\in E^{(p)}}\delta^{\{n\}}\varphi_T+\frac1{2!}\sum_{n_1,n_2\in E^{(p)}\atop\text{distinct}}\delta^{\{n_1,n_2\}}\varphi_T+\ldots+\frac1{k!}\sum_{n_1,\ldots,n_k\in E^{(p)}\atop\text{distinct}}\delta^{\{n_1,\ldots,n_k\}}\varphi_T,
$$
which we rewrite in the more compact form
\begin{align}\label{eq:clusterphiT}
\varphi_T^{(p)}\leadsto \sum_{j=0}^k\sum_{F\subset E^{(p)}\atop|F|=j}\delta^F\varphi_T,
\end{align}
where $\sum_{|F|=j}$ denotes the sum over $j$-uplets of integers (when $j=0$, this sum reduces to the single term $F=\varnothing$).
Intuitively, this means that $\varphi_T^{(p)}$ is expected to be close to a series where the term of order $\ell$ involves a correction due to the $\ell$-particle interactions.

For convenience, we set $\delta^F_e\varphi_T^H:=\delta^F\varphi_T^H$ for $F\ne\varnothing$, and $\delta^\varnothing_e\varphi_T^H:=\varphi_T^H+e\cdot x$.
Using the binomial formula in form of $\sum_{S\subset E}(-1)^{|E\setminus S|}=0$ for $E\ne\varnothing$, we easily deduce
\begin{equation}\label{eq:defdeltaxi1}
\nabla \delta_e^G\varphi_T^{F\cup H} \,=\,\sum_{S\subset F}\nabla \delta_e^{S\cup G} \varphi_T^H.
\end{equation}

\medskip
\noindent
{\bf Inclusion-exclusion formula.}

When the inclusions $\{J_n\}_n$ are disjoint, we have
\begin{align}\label{eq:exclDISJ}
C^{(p)}=\sum_{n\in E^{(p)}} C^{\{n\}}.
\end{align}
However, since inclusions may overlap, intersections are accounted for several times in the right-hand side and
this formula no longer holds. We now recall a suitable system of notation to deal with those intersections.

For any (possibly infinite) subset $E\subset\N$, we set $A_E:=A_1+C_E$, where
$C_E:=(A_2-A_1)\mathds1_{J_E}$ and $J_E:=\bigcap_{n\in E}J_n$. Note that $J_{\{n\}}=J^{\{n\}}=J_n$ and $C^{\{n\}}=C_{\{n\}}$. 
For non-necessarily disjoint inclusions, $C^{(p)}$ is then given by the following general inclusion-exclusion formula:
\begin{align}\label{eq:excl}
C^{(p)}&=\sum_{n\in E^{(p)}} C_{\{n\}}-\sum_{n_1<n_2\in E^{(p)}}C_{\{n_1,n_2\}}+\sum_{n_1<n_2<n_3\in E^{(p)}}C_{\{n_1,n_2,n_3\}}-\ldots\nonumber\\
&=\sum_{k=1}^\infty(-1)^{k+1}\sum_{F\subset E^{(p)}\atop|F|=k}C_{F}.
\end{align}
Since the inclusions $J_n$'s have a bounded diameter and the point set is almost surely locally finite, the sum~\eqref{eq:excl} is locally finite almost surely. 

We shall need further notation in the proofs. For all $E,F\subset\N$, $E\ne\varnothing$, we set $J_{E\| F}:=(\bigcap_{n\in E}J_n)\setminus (\bigcup_{n\in F}J_n)$ and $J^E_{\| F}:=(\bigcup_{n\in E}J_n)\setminus (\bigcup_{n\in F}J_n)$, and then
\[C_{E\|F}:=(A_2-A_1)\mathds1_{J_{E\| F}}\qquad\text{and}\qquad C^E_{\|F}:=(A_2-A_1)\mathds1_{J^E_{\| F}}.\]
In particular, we have $C_{E\|\varnothing}=C_E$, $C^E_{\|\varnothing}=C^E$, and $C^\varnothing_{\|F}=0$. For simplicity of notation, we also set $C_{\varnothing\|F}=0=C_\varnothing$.
The inclusion-exclusion formula then yields for all $G,H\subset \N$ with $G\ne\varnothing$,
\begingroup\allowdisplaybreaks
\begin{eqnarray}
C^{H}&=&\sum_{S\subset H}(-1)^{|S|+1}C_S,\label{eq:excl3.1}
\\
C^{H}_{\|G}&=&\sum_{S\subset H}(-1)^{|S|+1}C_{S\|G},\label{eq:excl3.2}
\\
C_{G\|H}&=&\sum_{S\subset H}(-1)^{|S|}C_{S\cup G}\label{eq:excl3.3}.
\end{eqnarray}\endgroup

In \cite[Corollary~2.2]{DG-16a}, we established the following formulas for the coefficients $\{\bar A_{T,h}^{j}\}_j$ in~\eqref{eq:expansion-ATh}, which can be viewed as natural cluster formulas.
\begin{lem}\label{lem:cluster-form}
For all $T,h>0$, we have for all $j\ge  0$,
\begin{equation}\label{eq:formderdisj0}
e\cdot \bar A^{ j}_{T,h}e\,=\,j! \sum_{|F|=j}\sum_{G\subset F}(-1)^{|F\setminus G|+1}\E_h\left[\nabla\delta_e^G\varphi_{T,h} \cdot C_{F\setminus G\| G}(\nabla\varphi_{T,h}^{  F}+e)\right].
\qedhere
\end{equation}
\end{lem}

\subsection{Optimal $\ell^1-\ell^2$ estimates}
In~\cite{DG-16a}, we used the naming ``$\ell^1-\ell^2$ estimates'' for the following family of estimates, which state that sums can be pulled out of the square without changing the bounds. In the present Poisson setting, this statement is to be compared to~\cite[Proposition~5.3]{giunti2021smoothness}.

\begin{prop1}\label{prop:main}
There exists a constant $C<\infty$ such that for all $T,h>0$ and $j,k\ge0$,  
\begin{equation}\label{eq:boundS}
S_j^k:=\E_h\bigg[ \sum_{|G|=k}\Big|\sum_{|F|=j\atop F\cap G=\varnothing}\nabla\delta^{F\cup G}\varphi_{T,h}\Big|^2\bigg]\le j! C^{k+j}.
\qedhere
\end{equation}
\end{prop1}

As in \cite{giunti2021smoothness}, the proof combines the original arguments for \cite[Proposition~4.6]{DG-16a} together with the following interesting property of the Poisson point process.
\begin{lem}\label{lem:JC}
Let $R$ be a bounded random function of indices with $R(\varnothing)=0$, and assume that it is approximately local in the sense that there exists $\kappa>0$ such that for all $F$,
\begin{equation}\label{e.app-loc}
|R(F)|\lesssim \sum_{n\in F}e^{-\kappa|x_n|}.
\end{equation}
 Then there exists $C<\infty$ (depending only on $d$ and on our fixed intensity $\lambda$) such that for all $h>0$ and $a,b,c\ge 1$ we have 
\begin{equation}
\E_h\bigg[\sum_{|H|=a,|G|=b\atop H\cap G=\varnothing}\mathds1_{J_H}\Big|\sum_{|F|=c\atop F\cap (H\cup G)=\varnothing}R(F\cup G)\Big|^2\bigg]\, \le \, \frac{C^a}{a!}\E_h\bigg[\sum_{|G|=b}\Big|\sum_{|F|=c\atop F\cap G=\varnothing}R(F\cup G)\Big|^2\bigg].
\label{eq:lala}
\qedhere
\end{equation}
\end{lem}
\begin{proof}[Proof]
Because of approximate locality \eqref{e.app-loc}, the left-hand side of \eqref{eq:lala} is finite for all finite $a,b,c$, and we have 
\begin{multline*}
\lim_{\rho \uparrow \infty} \E_h\bigg[\sum_{|H|=a,|G|=b\atop H\cap G=\varnothing}\mathds1_{J_H}\Big|\sum_{|F|=c\atop F\cap (H\cup G)=\varnothing}R_\rho(F\cup G)\Big|^2\bigg]\\
=\,\E_h\bigg[\sum_{|H|=a,|G|=b\atop H\cap G=\varnothing}\mathds1_{J_H}\Big|\sum_{|F|=c\atop F\cap (H\cup G)=\varnothing}R(F\cup G)\Big|^2\bigg],
\end{multline*}
where $\{R_\rho\}_{\rho}$ stands for the finite-volume restrictions $R_\rho(F):=R(F \cap \{n:x_n\in Q_\rho\})$.
Hence it suffices to prove the claim for $R_\rho$ instead of $R$.
As $R_\rho(F)$ only depends on indices for points in $Q_\rho$, we may condition the expectation 
with respect to the number of points in $Q_\rho$, to the effect of
\begin{eqnarray*}
\lefteqn{\E_h\bigg[\sum_{|H|=a,|G|=b\atop H\cap G=\varnothing}\mathds1_{J_H}\Big|\sum_{|F|=c\atop F\cap (H\cup G)=\varnothing}R_\rho(F\cup G)\Big|^2\bigg]}
\\
&=&\sum_{n=a+b+c}^\infty\pr{\sharp \Pc_h\cap Q_\rho=n} \E_{h,\rho,n}\bigg[\sum_{|H|=a,|G|=b\atop H\cap G=\varnothing}\mathds1_{J_H}\Big|\sum_{|F|=c\atop F\cap (H\cup G)=\varnothing}R_\rho(F\cup G)\Big|^2\, \bigg],
\end{eqnarray*}
where $\E_{h,\rho,n}[\cdot]:=\E_h[\cdot\,|\, \sharp (\Pc_h \cap Q_\rho)=n]$.
The complete independence of $\Pc_h$ now ensures that $\E_{h,\rho,n}$ coincides with normalized integration on $Q_\rho$ with respect to all $n$ points.
This yields in particular
\begin{eqnarray*}
\lefteqn{\E_{h,\rho,n}\bigg[\sum_{|H|=a,|G|=b\atop H\cap G=\varnothing}\mathds1_{J_H}\Big|\sum_{|F|=c\atop F\cap (H\cup G)=\varnothing}R_\rho(F\cup G)\Big|^2\bigg]}\nonumber\\
&=&\binom{n}{a}\Big(\fint_{Q_\rho^a}\mathds1_{x_1,\ldots,x_a\in B}dx_1\ldots dx_a\Big) \E_{h,\rho,n-a}\bigg[\sum_{|G|=b}\Big|\sum_{|F|=c\atop F\cap G=\varnothing}R_\rho(F)\Big|^2\bigg]\nonumber\\
&\le&C^a\binom{n}{a}\rho^{-da} \E_{h,\rho,n-a}\bigg[\sum_{|G|=b}\Big|\sum_{|F|=c\atop F\cap G=\varnothing}R_\rho(F)\Big|^2\bigg].
\end{eqnarray*}
Noting that
\[\pr{\sharp(\Pc_h\cap Q_\rho)=n}\binom{n}{a}\rho^{-da}\lesssim\frac1{a!}\pr{\sharp(\Pc_h\cap Q_\rho)=n-a},\]
the claim now follows by summation in form of
\begingroup\allowdisplaybreaks
\begin{align*}
&\E_{h}\bigg[\sum_{|H|=a,|G|=b\atop H\cap G=\varnothing}\mathds1_{J_H}\Big|\sum_{|F|=c\atop F\cap (H\cup G)=\varnothing}R_\rho(F\cup G)\Big|^2\bigg]\nonumber\\
&~~=~~\sum_{n=a+b+c}^\infty\pr{\sharp(\Pc_h\cap Q_\rho)=n}\E_{h,\rho,n}\bigg[\sum_{|H|=a,|G|=b\atop H\cap G=\varnothing}\mathds1_{J_H}\Big|\sum_{|F|=c\atop F\cap (H\cup G)=\varnothing}R_\rho(F\cup G)\Big|^2\bigg]\nonumber\\
&~~\le~~C^a\sum_{n=a+b+c}^\infty\pr{\sharp(\Pc_h\cap Q_\rho)=n}\binom{n}{a}\rho^{-da}\E_{h,\rho,n-a}\bigg[\sum_{|G|=b}\Big|\sum_{|F|=c\atop F\cap G=\varnothing}R_\rho(F\cup G)\Big|^2\bigg]\nonumber\\
&~~\lesssim~~\frac{C^a}{a!}\E_h\bigg[\sum_{|G|=b}\Big|\sum_{|F|=c\atop F\cap G=\varnothing}R_\rho(F\cup G)\Big|^2\bigg].\qedhere
\end{align*}
\endgroup
\end{proof}
With the above lemma at hand, we are in position to prove Proposition~\ref{prop:main}.
\begin{proof}[Proof of Proposition~\ref{prop:main}]
The proof closely follows that of~\cite[Proposition~4.6]{DG-16a}.
In particular, it is based on a double induction argument in $j$ and $k$.
The only difference with the original proof of  \cite[Proposition~4.6]{DG-16a} is that we appeal to Lemma~\ref{lem:JC} each time we need to control a term of the form $\sum_{U\subset G}\mathds1_{J_{U}}$ (which is uniformly
bounded if the point process is uniformly locally finite).

\medskip

\step1 General recurrence relation.\\
Let $G\subset\N$ be a finite subset. Summing the equation satisfied by $\delta_e^{F\cup G}\varphi_{T,h}$ over $F$, cf.~\cite[Lemma~4.1]{DG-16a}, we find
\begin{eqnarray*}
\lefteqn{\frac1T\sum_{|F|=j+1\atop F\cap G=\varnothing}\delta_e^{F\cup G}\varphi_{T,h} -\nabla\cdot\Aa\nabla\sum_{|F|=j+1\atop F\cap G=\varnothing}\delta_e^{F\cup G}\varphi_{T,h}}\\
&=&\nabla\cdot\sum_{|F|=j+1\atop F\cap G=\varnothing}\sum_{S\subset F}\sum_{U\subset G}(-1)^{|S|+|U|+1}C_{S\cup U\| G\setminus U}\nabla\delta_e^{(F\setminus S)\cup(G\setminus U)}\varphi_{T,h}^S\\
&=&\nabla\cdot\sum_{U\subset G}\sum_{S\le j+1\atop S\cap G=\varnothing}(-1)^{|S|+|U|+1}C_{S\cup U\| G\setminus U}\sum_{|F|=j+1-|S|\atop F\cap(G\cup S)=\varnothing}\nabla\delta_e^{F\cup(G\setminus U)}\varphi_{T,h}^S.
\end{eqnarray*}
The energy estimate then yields after summing over $G$ (see e.g.~\cite[proof of Lemma~4.2]{DG-16a}),
\begin{multline*}
S_{j+1}^{k+1}\,:=\,\E_h\bigg[\sum_{|G|=k+1}\Big|\nabla\sum_{|F|=j+1\atop F\cap G=\varnothing}\delta_e^{F\cup G}\varphi_{T,h}\Big|^{2}\bigg]\\
\,\lesssim\,\E_h\bigg[\sum_{|G|=k+1}\Big|\sum_{U\subset G}\sum_{|S|\le j+1\atop S\cap G=\varnothing}(-1)^{|S|+|U|+1}C_{S\cup U\| G\setminus U}\sum_{|F|=j+1-|S|\atop F\cap(G\cup S)=\varnothing}\nabla\delta_e^{F\cup(G\setminus U)}\varphi_{T,h}^S\Big|^{2}\bigg].
\end{multline*}
Since we have $|C_{S\cup U\|G\setminus U}|\lesssim\mathds1_{J_S}\mathds1_{J_{U\|G\setminus U}}$, and since the family $\{J_{U\|G\setminus U}\}_{U\subset G}$ is disjoint for fixed $G$, we deduce
\begin{equation}\label{eq:plop}
S_{j+1}^{k+1}
\,\lesssim\,\E_h\bigg[\sum_{|G|=k+1}\sum_{U\subset G}\mathds1_{J_{U}}\bigg(\sum_{|S|\le j+1\atop S\cap G=\varnothing}\mathds1_{J_S}\Big|\sum_{|F|=j+1-|S|\atop F\cap(G\cup S)=\varnothing}\nabla\delta_e^{F\cup(G\setminus U)}\varphi_{T,h}^S\Big|\bigg)^{2}\bigg].
\end{equation}
Using the decomposition $\nabla\delta_e^{F\cup(G\setminus U)}\varphi_{T,h}^S=\sum_{R\subset S}\nabla\delta_e^{F\cup (G\setminus U)\cup R}\varphi_{T,h}$, cf.~\eqref{eq:defdeltaxi1}, this leads to
\begin{equation*}
S_{j+1}^{k+1}
\,\lesssim\,\E_h\bigg[\sum_{|G|=k+1}\sum_{U\subset G}\mathds1_{J_{U}}\bigg(\sum_{|S|\le j+1\atop S\cap G=\varnothing}\sum_{R\subset S}\mathds1_{J_S}\Big|\sum_{|F|=j+1-|S|\atop F\cap(G\cup S)=\varnothing}\nabla\delta_e^{F\cup(G\setminus U)\cup R}\varphi_{T,h}\Big|\bigg)^{2}\bigg],
\end{equation*}
or alternatively, disjointifying the sets,
\begin{multline}\label{eq:plop0}
S_{j+1}^{k+1}
\,\lesssim\,
\sum_{\alpha=1}^{k+1}\E_h\bigg[\sum_{|G|=k+1-\alpha,|U|=\alpha\atop G\cap U=\varnothing}\mathds1_{J_{U}}\bigg({\sum_{\beta={1}}^{j+1}\sum_{\gamma=0}^\beta}\sum_{|S|=\beta-\gamma,|R|=\gamma\atop (S\cup R)\cap (G\cup U)=S\cap R=\varnothing}\mathds1_{J_{R\cup S}}\\
\times\Big|\sum_{|F|=j+1-\beta\atop F\cap(G\cup U\cup S\cup R)=\varnothing}\nabla\delta_e^{F\cup G\cup R}\varphi_{T,h}\Big|\bigg)^{2}\bigg].
\end{multline}
Using~\eqref{eq:lala} in~\eqref{eq:plop0} (which we can since the massive approximation makes the corrector gradient approximately local with $\kappa\simeq1/\sqrt{T}$), we get
\begin{multline}\label{eq:bas}
S_{j+1}^{k+1}
\,\lesssim\,
\sum_{\alpha={1}}^{k+1}
\frac{C^\alpha}{\alpha !}
 \E_h\bigg[\sum_{|G|=k+1-\alpha}\bigg({\sum_{\beta={1}}^{j+1}\sum_{\gamma=0}^\beta}\sum_{|S|=\beta-\gamma,|R|=\gamma\atop (S\cup R)\cap G=S\cap R=\varnothing}\mathds1_{J_{R\cup S}}\\
\times
\Big|\sum_{|F|=j+1-\beta\atop F\cap(G\cup S\cup R)=\varnothing}\nabla\delta_e^{F\cup G\cup R}\varphi_{T,h}\Big|\bigg)^{2}\bigg].
\end{multline}
Now expanding the square,
\begin{multline*}
\E_h\bigg[\sum_{|G|=k+1-\alpha}\bigg({\sum_{\beta={1}}^{j+1}\sum_{\gamma=0}^\beta}\sum_{|S|=\beta-\gamma,|R|=\gamma\atop (S\cup R)\cap G=S\cap R=\varnothing}\mathds1_{J_{S\cup R}}
\Big|\sum_{|F|=j+1-\beta\atop F\cap(G\cup S\cup R)=\varnothing}\nabla\delta_e^{F\cup G\cup R}\varphi_{T,h}\Big|\bigg)^2\bigg]\\
\,=\,{\sum_{\beta=1}^{j+1}\sum_{\gamma=0}^\beta\sum_{\beta'=1}^{j+1}\sum_{\gamma'=0}^{\beta'}}\E\bigg[\sum_{|G|=k+1-\alpha}\sum_{|S|=\beta-\gamma,|R|=\gamma\atop (S\cup R)\cap G=S\cap R=\varnothing}\sum_{|S'|=\beta'-\gamma',|R'|=\gamma'\atop (S'\cup R')\cap G=S'\cap R'=\varnothing}{\mathds1_{J_{S\cup R}}\mathds1_{J_{S'\cup R'}}}\\
\times\Big|\sum_{|F|=j+1-\beta\atop F\cap(G\cup S\cup R)=\varnothing}\nabla\delta_e^{F\cup G\cup R}\varphi_{T,h}\Big|\Big|\sum_{|F'|=j+1-\beta'\atop F\cap(G\cup S'\cup R')=\varnothing}\nabla\delta_e^{F\cup G\cup R'}\varphi_{T,h}\Big|\bigg],
\end{multline*}
and {making $F$ (resp.~$F'$) disjoint from $S',R'$ (resp.~$S,R$) in form of
\[\Big|\sum_{|F|=j+1-\beta\atop F\cap(G\cup S\cup R)=\varnothing}\nabla\delta_e^{F\cup G\cup R}\varphi_{T,h}\Big|\le
\sum_{S_0'\subset S',R_0'\subset R'}\Big|\sum_{|F|=j+1-\beta-|S_0'|-|R_0'|\atop F\cap(G\cup S\cup R\cup S'\cup R')=\varnothing}\nabla\delta_e^{F\cup G\cup R\cup S_0'\cup R_0'}\varphi_{T,h}\Big|,\]
we deduce, using the bounds $ab\le a^2+b^2$ and $\sum_{H'\subset H}1\le2^{|H|}$,}
\begingroup\allowdisplaybreaks
\begin{multline*}
\E_h\bigg[\sum_{|G|=k+1-\alpha}\bigg({\sum_{\beta={1}}^{j+1}\sum_{\gamma=0}^\beta}\sum_{|S|=\beta-\gamma,|R|=\gamma\atop (S\cup R)\cap G=S\cap R=\varnothing}\mathds1_{J_{S\cup R}}
\Big|\sum_{|F|=j+1-\beta\atop F\cap(G\cup S\cup R)=\varnothing}\nabla\delta_e^{F\cup G\cup R}\varphi_{T,h}\Big|\bigg)^2\bigg]\\
\,\lesssim\, {\sum_{\beta=1}^{j+1}\sum_{\gamma=0}^\beta\sum_{\beta'=1}^{j+1}\sum_{\gamma'=0}^{\beta'}}{2^{\beta'}\,} \E\bigg[\sum_{|G|=k+1-\alpha}\sum_{|S|=\beta-\gamma,|R|=\gamma\atop (S\cup R)\cap G=S\cap R=\varnothing}\sum_{|S'|=\beta'-\gamma',|R'|=\gamma'\atop (S'\cup R')\cap G=S'\cap R'=\varnothing}\mathds1_{J_{S\cup R}}\mathds1_{J_{S'\cup R'}}\\
\times\sum_{S_0'\subset S',R_0'\subset R'}\Big|\sum_{|F|=j+1-\beta-|S_0'|-|R_0'|\atop F\cap(G\cup S\cup R\cup S'\cup R')=\varnothing}\nabla\delta_e^{F\cup G\cup R\cup S_0'\cup R_0'}\varphi_{T,h}\Big|^2\bigg].
\end{multline*}
\endgroup
As all sums are on disjoint index sets, we are now in position to appeal again to~\eqref{eq:lala}, and we easily deduce after straightforward simplifications,
\begin{multline*}
\E_h\bigg[\sum_{|G|=k+1-\alpha}\bigg({\sum_{\beta={1}}^{j+1}\sum_{\gamma=0}^\beta}\sum_{|S|=\beta-\gamma,|R|=\gamma\atop (S\cup R)\cap G=S\cap R=\varnothing}\mathds1_{J_{S\cup R}}
\Big|\sum_{|F|=j+1-\beta\atop F\cap(G\cup S\cup R)=\varnothing}\nabla\delta_e^{F\cup G\cup R}\varphi_{T,h}\Big|\bigg)^2\bigg]\\
\,\lesssim\,{\sum_{\beta=1}^{j+1}\sum_{\gamma=0}^\beta\sum_{\beta'=1}^{j+1}\sum_{\gamma'=0}^{\beta'}\sum_{\delta=0}^{\beta'}{\frac{C^{\beta+\beta'-\gamma-\delta}}{(\beta+\beta'-\gamma-\delta)!}}}\,\E_h\bigg[\sum_{|G|=k+1-\alpha+\gamma+\delta}\Big|\sum_{|F|=j+1-\beta-\delta\atop F\cap G=\varnothing}\nabla\delta_e^{F\cup G}\varphi_{T,h}\Big|^2\bigg].
\end{multline*}
Inserting this into~\eqref{eq:bas}, and noting that $\frac{1}{m!n!}C^mC^n\le\frac1{(m+n)!}(2C)^{m+n}$, we are then led to
\begin{multline*}
S_{j+1}^{k+1}
\,\lesssim\,
{\sum_{\alpha=1}^{k+1}\sum_{\beta=1}^{j+1}\sum_{\gamma=0}^\beta\sum_{\beta'=1}^{j+1}\sum_{\gamma'=0}^{\beta'}\sum_{\delta=0}^{\beta'}{\frac{C^{\alpha+\beta+\beta'-\gamma-\delta}}{(\alpha+\beta+\beta'-\gamma-\delta)!}}}\\
\times\E_h\bigg[\sum_{|G|=k+1-\alpha+\gamma+\delta}\Big|\sum_{|F|=j+1-\beta-\delta\atop F\cap G=\varnothing}\nabla\delta_e^{F\cup G}\varphi_{T,h}\Big|^2\bigg],
\end{multline*}
or equivalently, after reorganizing the sums,
\begin{equation}\label{e.induction}
S_{j+1}^{k+1}
\,\lesssim\,
\sum_{\delta=0}^{j+1}\sum_{\beta=1}^{j+1}\sum_{\alpha=0}^{k+\beta}{\frac{C^{k+\beta-\alpha+1}}{(k+\beta-\alpha+1)!}}
 S^{\alpha+\delta}_{j+1-\beta-\delta}.
\end{equation}

\medskip

\step2 Conclusion.\\
We initialize the induction by noting that
$$
S_0^0 \le C,
$$
which is  nothing but the standard energy estimate for the corrector $\varphi_T$ (an a priori estimate that only requires the uniform ellipticity of $A$).  
Then, by a similar (double) induction argument as in~\cite{DG-16a}, now based on~\eqref{e.induction}, the claim follows (for some possibly different constant $C<\infty$).
\end{proof}

\subsection{Proof of Gevrey regularity}

The rest of the proof follows our general argument in~\cite{DG-16a}.
First, adapting the proof of \cite[Proposition~5.2]{DG-16a} by using Lemma~\ref{lem:JC} (as we did above for \cite[Proposition~4.6]{DG-16a}), and replacing  \cite[Proposition~4.6]{DG-16a} by Proposition~\ref{prop:main}, we directly obtain the uniform bounds \eqref{e.unif-bd}. In order to use this bound to prove regularity based on the qualitative convergence \eqref{e.qual-conv} and the regularity of $p \mapsto \bar A_{T,h}^{(p)}$, it remains to appeal 
to a Taylor formula in form of~\cite[(5.25)]{DG-16a}: for all $k$ and $p\in [0,1]$,
$$
|\bar A^{(p)}_{T,h}- \sum_{j=0}^k \frac{p^j}{j!} \bar A^{j}_{T,h}| \le \frac{p^{k+1}}{(k+1)!} \sup_{u \in [0,p]}  |\bar A^{k+1}_{T,h}(\Pc^{(u)})|,
$$
where $\bar A^{k+1}_{T,h}(\Pc^{(u)})$ denotes the $(k+1)$th term of the expansion associated with the (partially) decimated point process $\Pc^{(u)}$, which is itself in the present case a Poisson point process with intensity $\lambda u$, hence for which the  bound \eqref{e.unif-bd} holds uniformly on $u\in [0,p]$.
Since the constants are uniform wrt $T,h$, as in \cite{DG-16a}, this entails the existence of 
the limits $\bar A^j=\lim_{T\uparrow\infty,h\downarrow0}\bar A^j_{T,h}$, and there holds for all $j,k\ge 0$ and $p\in [0,1]$,
$$
\Big|\bar A^{(p)}- \sum_{j=0}^k \frac{p^j}{j!} \bar A^{j} \Big| \le (k+1)!\,(Cp)^{k+1}\qquad\text{and}
\qquad |\bar A^{j}| \le j!^2C^j.
$$
The conclusion of Theorem~\ref{th:main} then follows from the arguments at the beginning of Section~\ref{sec:strategy}.

\section*{Acknowledgements}
MD acknowledges financial support from the CNRS-Momentum program,
and AG from the European Research Council (ERC) under the European Union's Horizon 2020 research and innovation programme (Grant Agreement n$^\circ$~864066).

\bibliographystyle{plain}

\def\cprime{$'$} \def\cprime{$'$} \def\cprime{$'$}

\end{document}